\documentclass[a4paper,10pt]{article}

\usepackage{times}
\usepackage[T1]{fontenc}
\usepackage[latin1]{inputenc}
\usepackage{amsmath,amssymb}
\usepackage{amsthm}
\usepackage{amscd}

\newtheorem{theorem}{Theorem}[section]
\newtheorem{lemma}[theorem]{Lemma}

\newtheorem{proposition}[theorem]{Proposition}

\newtheorem{definition}[theorem]{Definition}

\theoremstyle{definition}
\newtheorem{example}[theorem]{Example}
\newtheorem{remark}[theorem]{Remark}

\numberwithin{table}{section}
\numberwithin{equation}{section}

\begin{document}
\title{ Sylvester equations and polynomial separation of spectra} 
\author{ Olavi Nevanlinna }
\maketitle

 \begin{center}
{\footnotesize\em 
Aalto University\\
Department of Mathematics and Systems Analysis\\
 email: Olavi.Nevanlinna\symbol{'100}aalto.fi\\[3pt]
}
\end{center}

\begin{abstract}
Sylvester  equations $AX-XB=C$ have unique solutions for all $C$ when the spectra of $A$ and $B$ are disjoint. Here $A$ and $B$ are bounded operators in Banach spaces.  We discuss the existence of  polynomials $p$ such that the spectra of $p(A)$ and $p(B)$ are well separated,  either inside and outside of a circle or separated into different half planes.    Much of the discussion is based on  the following inclusion sets for the spectrum:   $V_p(T)=\{\lambda \in \mathbb C \ : \ |p(\lambda)| \le \|p(T)\| \}$ where $T$ is a bounded operator.  We also give an explicit series expansion for the solution in  terms of  $p(M)$, where $M=\begin{pmatrix} A&C\\ &B\end{pmatrix}$,  in the case where the spectra of $A$ and $B$ lie in different components of $V_p(M)$ .

\end{abstract}

\bigskip
{\it Key words:} Sylvester equation, multicentric calculus, preconditioning, spectral separation

{\it 2010 Mathematics Subject Classification:}  15A24, 47A10, 47A60, 47A62, 65F08, 65F10, 65J10

\bigskip
 
 \section{Introduction}
 
 We discuss the solution of the Sylvester equation
 \begin{equation}\label{original}
 AX-XB=C
 \end{equation}  
 by  solving first a related equation
 \begin{equation}\label{modified}
 p(A)Y-Yp(B)=C
 \end{equation}
 which is assumed to be easier to solve and then recover the solution of (\ref{original}) as
 \begin{equation}\label{solution}
 X= q(A,B)(Y).
 \end{equation}
 Here the operator  $q(A,B)$ is obtained by the bivariate polynomial functional calculus from the divided difference of $p$,  see Section 2, below. Alternatively, one can first form a new right hand side and  consider solving
$$
p(A)X-Xp(B)= q(A,B)(C),
$$
see Propositions 2.2 and 2.4.

We consider the equations in the generality of bounded operators in Banach spaces. Given  Banach spaces $\mathcal X, \mathcal Y$  we assume that $A$ is bounded in $\mathcal X$, $B$ in  $\mathcal Y$ and  while $C$ and the unknowns $X$ and $Y$ are bounded operators from  $\mathcal Y$ to $\mathcal X$.   We  discuss solution methods which can be formulated in infinite dimensional cases but which  should be useful in matrix  problems, in particular when the dimensions are large so that direct methods may not be practical.   In this introduction we mention two basic representations for the solution, and then provide the spectral conditions under which a  polynomial $p$ exist so that  these  methods can be used.

 In a series of papers [10,11,12] we have studied the possibility of  taking a polynomial as a new {\it global} variable.  As  polynomials are not injective  we represent scalar functions $\varphi : z \mapsto \varphi(z) \in \mathbb C$ by vector valued functions $f: w \mapsto f(w) \in \mathbb C^d$ where $w=p(z)$ and $p$ is a polynomial of degree $d$ with simple roots $\lambda_j$.   Then $\varphi$ is represented in the {\it multicentric} form
\begin{equation}\label{multi}
\varphi(z) = \sum_{j=1}^d \delta_j(z) f_j(p(z))
\end{equation}
where $\delta_j$ is the Lagrange polynomial   $\delta_j(z) = \prod_{k\not=j}\frac{z-\lambda_k}{\lambda_j-\lambda_k}$.  In this representation  $\delta_j(A)$ is always well defined for any bounded operator and if $p(A)$ is "simpler" than $A$, small in norm,  diagonalizable, normal, etc,  an efficient functional calculus may be available for  defining and computing $f_j(p(A))$.  

Here  the idea is again to replace the operators $A$ and $B$   by $p(A)$ and $p(B)$ but  part of our dicussion is independent of the multicentric  calculus.
However, we discuss an application of the multicentric calculus which can be viewed as a modification of the {\it sign-}function approach, leading to a series expansion   given  in  powers of  $p(M)$ where $M=\begin{pmatrix}A&C\\ &B\end{pmatrix}$.

We shall now summarize the key results on the Sylvester equation, needed in the following.      If $T$ is a bounded operator in a Banach space, then we denote by $\sigma(T)$   the spectrum:
$$
\sigma(T)=\{ \lambda \in \mathbb C \ :  \lambda-T  \text { is not invertible} \}.$$

Bhatia and Rosenthal have written a readable  survey of (\ref{original}), [1]. They call the  following as Sylvester-Rosenblum  Theorem.

\begin{theorem}  Let $\mathcal X$ and $\mathcal Y$ be Banach spaces and $A$, $B$ bounded operators in $\mathcal X$ and $\mathcal Y$, respectively.
If 
\begin{equation}\label{exist}
\sigma(A) \cap \sigma(B) = \emptyset,
\end{equation}
 then the equation (\ref{original}) has a unique solution $X \in \mathcal B(\mathcal Y, \mathcal X)$ for every  $C\in\mathcal B(\mathcal Y, \mathcal X)$.  \end{theorem}

We shall only consider the cases where (\ref{exist}) holds.  Thus at least one of the operators $A$ and $B$ can be assumed to be nonsingular, and we shall assume  that $B$ is.  This is no restriction of generality as we could "transpose" the equation.  Further, if  $\lambda$ is a regular point for both $A$ and $B$ we could consider the equivalent equation
\begin{equation}\label{shifted}
(A-\lambda)X-X(B-\lambda)=C
\end{equation}
instead and then both operators are invertible.  This leads to the following representation of the solution.  

\begin{theorem}\label{ros}({\rm [14]})
If $\gamma$ is a union of closed contours  with total winding numbers 1 around $\sigma(A)$ and 0 around $\sigma(B)$, then the solution of  (\ref{original}) can be expressed as
\begin{equation}\label{integral}
X= \frac{1}{2\pi i}\int_\gamma (\lambda-A)^{-1}C (\lambda -B)^{-1} d\lambda.
\end{equation}  
\end{theorem}
\begin{proof} Operate  (\ref{shifted}) by $(\lambda-A)^{-1}$ from left and with $(\lambda-B)^{-1}$ from right. Integrating over $\gamma$ yields the claim.\end{proof}
Denote by $\rho(T)$ the spectral radius of $T$:    $\rho(T)= \sup\{ | \lambda| \ : \lambda \in \sigma(T)\}$.

\begin{proposition}
Assume  that $B$ is invertible and that 
$\rho(A) \rho(B^{-1}) < 1$.   Then   the series $\sum_{n=0}^\infty A^n C B^{-n-1}$ converges and setting
\begin{equation}
X= - \sum_{n=0}^\infty A^n C B^{-n-1}
\end{equation}
we have  a representation for the solution.
\end{proposition}
\begin{proof}
The series converges as  $$  \|A^n\| ^{1/n} \| CB^{-1} \|^{1/n} \|B^{-n}\|^{1/n}\rightarrow \rho(A)\rho(B^{-1})<1.$$  Multiplying the series by $A$ from left and subtracting the result of multiplying the series by $B$ from right then yields the claim.

Notice that this also follows  from Theorem \ref{ros} since by assumption  there exists an $r>0$ such that $\rho(A)<r$ and $\rho(B^{-1} ) < 1/r$. Then  we can integrate along $|\lambda|=r$  substituting
$$
(\lambda-A)^{-1} = \sum_{n=0}^\infty \lambda^{-n-1}A^n  \  \text  {  and } \ 
(\lambda-B)^{-1} = - \sum_{n=0}^\infty \lambda^n B^{-n-1}.
$$

\end{proof}

Our first aim is to discuss whether   for given $A$ and $B$ there is a polynomial $p$ such that
\begin{equation}\label{polcondition}
\rho(p(A)) \ \rho(p(B^{-1})) <1
\end{equation}
so that (\ref{modified})  could be solved as
 \begin{equation}
Y= - \sum_{n=0}^\infty p(A)^n C p(B)^{-n-1}.
\end{equation}
 
Recall, that  {\it the polynomially convex hull} $\widehat K$ of a compact set $K \subset
\mathbb C$ is defined as
\begin{equation}\label{polcondef}
\widehat K = \{ z \in \mathbb C \ : \ |p(z)| \le \|p\|_K \ \text{ for all  polynomials } p\}
\end{equation}
where $\|p\|_K= \sup_{z\in K}|p(z)|$. Thus $\widehat K$ is obtained by "filling the holes" of $K$. 
We have the following.

\begin{theorem}
There exists a polynomial $p$ such that $p(B)$ is invertible and (\ref{polcondition}) holds
 if and only if 
\begin{equation}\label{firstcond}
\widehat{\sigma(A)} \cap \sigma(B)=\emptyset.
\end{equation}
\end{theorem}
The proof  is in Section 3 where we also  show  how small  the product in (\ref{polcondition}), when properly normalized,  can be.  

The second aim concerns another sufficient  condition, based on the separation of the spectra of $A$ and $B$ by a vertical line.  
Again, by subtracting  a suitable constant from the operators we may assume that the line is the imaginary axis.  We shall denote by $\mathbb C_{+}$ the open right half plane and by $\mathbb C_{-}$ the  open left half plane.   

\begin{theorem} ({\rm [5]})  Suppose that the operators $A$, $B$ and $C$ are all bounded and that $\sigma(A) \subset \mathbb C_{+}$ and $\sigma(B) \subset \mathbb C_{-}$.  Then the solution of (\ref{original}) can be represented as
\begin{equation}\label{heinzintegral}
X = \int_0^\infty e^{-tA} C e^{tB} dt.
\end{equation}
\end{theorem}
\begin{proof}  For a small enough $\varepsilon>0$ and large enough $K$ we have for $t>0$
$$
\|e^{-tA}Ê\| \le K e^{-\varepsilon t} \  \text{ and } \|e^{tB}Ê\| \le K e^{-\varepsilon t}.
$$
Thus, the integral converges and the claim follows by  operating with $A$ from left and  integrating by parts.
\end{proof}

Recall that under the assumptions of Theorem 1.5 the  {\it sign}-function of the block operator $M$ is well defined and can be used to solve the Sylvester equation, see (\ref{signinkautta}).  On the possibility of separation into half planes we have the following result with proof in   Section 4.

\begin{theorem} There exists a polynomial $p$ such that
\begin{equation}\label{eripuolin}
\sigma(p(A)) \subset\mathbb C_{+}  \text { and } \ \sigma(p(B)) \subset \mathbb C_{-}
\end{equation}
if and only if 
\begin{equation}\label{wideintersection}
\widehat{\sigma(A)} \cap \widehat{\sigma(B)} =\emptyset
\end{equation}
 holds.
\end{theorem}

While (\ref{polcondition}) and  (\ref{eripuolin}) give  the conditions under which these separating polynomials exist,  one should expect that replacing the spectra by $\varepsilon$-pseudospectra  should give  useful information on  the difficulty of computing these polynomials. Denoting the
\begin{equation}\label{pseudo}
\Sigma_\varepsilon(T) = \{ \lambda\in \mathbb C \ : \  \text{ either } \lambda \in \sigma(T)  \text{ or } \|(\lambda-T)^{-1}\|  \ge \frac{1} {\varepsilon} \}
\end{equation}
we could ask for  how large $\varepsilon$ the conditions
$$
\widehat{\Sigma_\varepsilon(A)} \cap \Sigma_\varepsilon(B)=\emptyset \ \text{ and } \widehat{\Sigma_\varepsilon(A)} \cap \widehat{\Sigma_\varepsilon(B)}=\emptyset
$$
would hold.  However,  it seems that a more useful concept in this connection is the following inclusion set
\begin{equation}\label{veepee}
V_p(T) = \{ \lambda \in \mathbb C \ : \ |p(\lambda)| \le \| p(T)\|  \}
\end{equation}
where $p$ is a polynomial.   For  (\ref{polcondition}) we would look for a polynomial $p$ such that
$$
V_p(A) \cap \sigma(B)=\emptyset 
$$
while  for (\ref{eripuolin})  we would look for a polynomial such that 
$V_p(A\oplus B)$ separates into different components, containing $\sigma(A)$  and $\sigma(B)$, respectively.  

In the  practical search  for   separating polynomials,  Krylov methods can be uselful, but  one cannot in general guarantee that they would always produce separating polynomials when the necessary and sufficent spectral conditions hold. However,   an idealized procedure exists  with guaranteed performance.   It assumes that  one can perform minimizations of norms at polynomials of the operator and the key point is that one need not to know  about the spectrum in advance.  The following is Theorem 1.3 in [9], see also [4].

\begin{theorem}
There exists a procedure  which, given $A\in \mathcal B(\mathcal X)$, produces a sequence of compact sets $K_k \subset \mathbb C$ and  polynomials $p_k$ satisfying the following:  $K_{k+1} \subset K_k$, $ V_{p_k}(A) \subset K_k$, and
$$
\widehat{\sigma(A)} = \bigcap_{k\ge1} K_k.
$$
\end{theorem}

In Section 2 we show how the post-processing is done.  Sections 3 and 4 contain proofs of Theorems 1.4 and 1.6  and  refinement of these. 
 
At the end in Section 5 we take a somewhat different approach.   We assume that  we have a polynomial $p$ such that $V_p(M)$ separates into two components in which we define a piecewise constant holomorphic function.   Using multicentric representation of this function we obtain a series expansion  in terms of $p(M)$  from which the  solution for the Sylvester equation can  be read out in the same way as from sgn$(M)$.  The coefficients of the series expansions can be computed with an explicit recursion depending on the polynomial $p$.

  
 \section{Post-processing}
 
 Assume  that  one has in one way or another solved the modified equation (\ref{modified}).  We assume that we know the operators $A\in \mathcal B(\mathcal X),B\in \mathcal B(\mathcal Y)$ and $Y\in \mathcal B(\mathcal Y,\mathcal X)$ and the  (scalar) polynomial $p$.  We shall use the bivariate polynomial calculus to write down the solution $X$ satisfying (\ref{original}).  To that end we associate with $p$ the bivariate  polynomial $q$ as the divided difference of $p$:
 \begin{equation}\label{defkuu}
 q(\lambda, \mu) = \frac{p(\lambda)-p(\mu)}{\lambda-\mu}.
 \end{equation}
 Denote  $q_{k-1}(\lambda, \mu)= \lambda^{k-1} + \lambda^{k-2}Ê\mu + \cdots + \mu^{k-1} $  with $q_0=1$.
 Since $\lambda^k-\mu^k = (\lambda - \mu) q_{k-1}(\lambda,\mu)$     we  then  have with $p(\lambda)= \sum_{j=0}^d \alpha_j \lambda^j$
\begin{equation}
q(\lambda,\mu)=\sum_{j=1}^d \alpha_j q_{j-1}(\lambda,\mu).
\end{equation}
On bivariate holomorphic functional calculus we recommend [7]. Since we deal here only with polynomials we can give the calculus without reference to integral  representations.  In the notation of [7],  $q\{A,B^T\} (C)$   stands for our $q(A,B)(C)$.  

\begin{definition}  Let the operators $A\in \mathcal B(\mathcal X),B\in \mathcal B(\mathcal Y)$ and $C\in \mathcal B(\mathcal Y,\mathcal X)$ and  the  polynomial $f(\lambda, \mu)= \sum_{i, j} \alpha_{ij} \lambda^{i} \mu^j$ be given.  Then we denote by $f(A,B)$  the bounded linear operator in $\mathcal B(\mathcal Y,  \mathcal X)$:
\begin{equation}\label{operdef}
f(A,B): \ C \mapsto  f(A,B)(C)= \sum_{i,j} \alpha_{ij} A^{i}C B^j.
\end{equation}

\end{definition}
When  $f$ is holomorphic in two variables one defines  $f(A,B)$ using a double integral and  based on that one can prove that if $h(\lambda,\mu) = g(\lambda,\mu)f(\lambda,\mu)$ one gets $h(A,B)(C) = g(A,B)(f(A,B)(C))$.  For polynomials   this is obvious from (\ref{operdef}) as we may work  termwise.  If $g(\lambda,\mu)=\lambda^m \mu^n$, $f(\lambda,\mu)=\lambda^{i} \mu^j$ then $g(\lambda,\mu)f(\lambda, \mu) = \lambda^{ i+m}\mu^{j+n} =h(\lambda, \mu)$ and   we have
$$
g(A,B)(f(A,B)(C)) = A^m(A^{i} C B^j)B^n = A^{m+i}CB^{n+j} = h(A,B)(C).
$$
Taking linear combinations  we  see that   $h(A,B) = g(A,B)\circ f(A,B)$ holds for polynomials $f$, $ g$ where $h= g f$.

Consider now  the post-processing step which is contained in the following simple result. 

\begin{proposition}   Let $A\in \mathcal B(\mathcal X),B\in \mathcal B(\mathcal  Y)$ and $Y\in \mathcal B(\mathcal Y,\mathcal X)$ be given and  a polynomial $p$, such that (\ref{modified}) holds.  Then
\begin{equation}
X=q(A,B)(Y)
\end{equation}
satisfies  the original Sylvester equation (\ref{original}).
\end{proposition}
\begin{proof}
We have
$$
p(\lambda)-p(\mu) = (\lambda-\mu) q(\lambda, \mu).
$$
 
Taking the left hand side as a polynomial of two variables  and applying the polynomial functional calculus yields, by (\ref{modified}),   $p(A)Y-Yp(B)=C$.  Now  the right hand side  gives 
$ A q(A,B)(Y) - q(A,B)(Y) B = AX-XB$ , completing the proof.

\end{proof}

\begin{example} Let $A$ be a nonsingular real {\it symmetric} matrix,  $B$  a real {\it skew symmetric} one.  Then $A^2$ is positive definite while $B^2$ is negative semidefinite and 
\begin{equation}\label{skewcase}
Y= \int_0^\infty e^{-tA^2} C e^{tB^2}dt
\end{equation}
solves the modified equation.  Now $q(\lambda, \mu) = \lambda +\mu$ and we have  the solution of the original Sylvester equation as
$$
X=q(A,B)(Y)= AY +YB = \int_0^\infty  (Ae^{-tA^2}C e^{tB^2}  + e^{-tA^2}C e^{tB^2}B )dt.
$$
The simple choice,  $p(\lambda)= \lambda^2$ works naturally  in a somewhat lager set of matrices.  In fact, if   there exists $\theta <1$ such that  if $\alpha+i\beta \in \sigma(A)$ then
$|\beta| \le \theta |\alpha|$ while with $\gamma + i \delta \in \sigma(B)$ we ask for $|\gamma| \le \theta |\delta|$.  If at least one of $A$ or $B$ is nonsingular, then again  
the integral in (\ref{skewcase}) converges.
\end{example}

Denoting $S(\lambda,\mu)=\lambda-\mu$ the solution operator is the inverse of $S(A,B)$   satisfying
\begin{equation}\label{inversiokompositio}
S(A,B)^{-1}=q(A,B) \circ S(p(A),p(B))^{-1}.
\end{equation}

Extending the bivarite polynomial calculus to holomorphic calculus one can show that if $f,g$ are holomorphic in two variables near the spectra and $h=gf$, then 
\begin{equation}
g(A,B)\circ f(A,B)=h(A,B),
\end{equation}
see e.g. Lemma 4.2 in [7].  Assuming this allows us to commute the terms in
(\ref{inversiokompositio}) and we conclude that rather than post-processing with $q(A,B)$ we may equally well begin with processing $C$.  Clearly the order of computation is not the same but  the operations needed to be excecuted essentially are.  To summarise:

\begin{proposition} Let $A\in \mathcal B(\mathcal X),B\in \mathcal B(\mathcal  Y)$ and $C\in \mathcal B(\mathcal Y,\mathcal X)$ be given and  a polynomial $p$ such that $\sigma(p(A)) \cap \sigma(p(B)) = \emptyset$.  Then 
\begin{equation}
p(A)X-Xp(B)= q(A,B)(C)
\end{equation}
has a unique solution $X$ which also satisfies (\ref{original}).
\end{proposition}


\section{Disc separation}

As before, $A \in \mathcal B(\mathcal X)$ and $B\in \mathcal B(\mathcal Y)$  and here we  consider the  convergence condition $\rho(p(A)) \  \rho(p(B)^{-1} )<1.$  
Theorem 1.4 covers the existence  of such polynomials and we give the proof here.  We also  derive an expression for the  normalized infimum of  the product of spectral radii.   At the end of this section we discuss  a more quantitative result.

If the spaces are finite dimensional, or more generally, if $A$ is an algebraic operator, then there exists a  minimal polynomial $m_A $ such that $m_A(A)=0$, and  assuming  $\sigma(A)\cap \sigma(B)=\emptyset$,  then trivially $\rho(m_A(A)) \ \rho(m_A(B)^{-1})=0$.  However,  the degree of $m_A$ may be impractically high and computation of $m_A$ unstable.  

 \bigskip
  
{\it Proof of Theorem 1.4}

Suppose first that
$\lambda_0 \in \widehat{\sigma(A)} \cap \sigma(B)$ and let $p$ be a polynomial such that $p(B)$ is invertible. Then
$$
|p(\lambda_0)| \ge \min_{\mu \in \sigma(B)} |p(\mu)| = 1/ \rho(p(B)^{-1}).
$$
Since $\rho(p(A)) \ge |p(\lambda_0)|$ we have
$$
\rho(p(A)) \ \rho(p(B)^{-1}) \ge |p(\lambda_0)| |p(\lambda_0)|^{-1} =1
$$ 
and we see that the condition (\ref{firstcond})
 is necessary.

Assume then that (\ref{firstcond}) holds.  As $\widehat{\sigma(A)}$ and $\sigma(B)$ are both compact, there exists an open $U$ such that $\widehat{\sigma(A)} \subset U$ while $\sigma(B) \cap U = \emptyset.$  By Hilbert Lemniscate  Theorem, see e.g. Theorem 5.5.8 in [13], there exists a polynomial $p$ such that
\begin{equation}
{|p(z)|} >\|p\|_{\sigma(A)}  \  \text{ for } z\in \mathbb C \setminus U.
\end{equation}
Thus, in particular $$1/ \rho(p(B)^{-1})=\min_{\mu \in \sigma(B)} |p(\mu)|>\|p\|_{\sigma(A)} = \rho(p(A))$$
and so
$
\rho(p(A)) \  \rho(p(B)^{-1}) <1,
$
completing the proof.  $\hfill \Box$
\bigskip

In practical computation, the spectral radius $\rho(p(A))$  should rather be replaced by $\|p(A)\|$ 
and scaled properly. To that end put
\begin{equation}
\eta(A,B)= \inf (\|p(A)\| \|p(B)^{-1}\| ) ^{1/{deg(p)}}
\end{equation}
where the infimum is over all polynomials $p$.  
\begin{lemma}
We have 
\begin{equation}\label{movingtospectral}
\eta(A,B) = \inf (\rho(p(A)) \  \rho(p(B)^{-1}))^{1/{deg(p)}}.
\end{equation}
\end{lemma}
\begin{proof}
The claim follows from the spectral radius formula.  In fact, given $\varepsilon>0$ there exists a polynomial $q$ of degree $k$ such that
$$
(\rho(q(A)) \ \rho(q(B)^{-1})^{1/k} <  \inf (\rho(p(A)) \  \rho(p(B)^{-1}))^{1/{deg(p)}} + \varepsilon.
$$
But we have  as $n\rightarrow \infty$
$$
\|q(A)^n\|^{1/kn} \|q(B)^{-n}\|^{1/kn} \rightarrow (\rho(q(A)) \ \rho(q(B)^{-1})^{1/k} 
$$
so that $\eta(A,B)$ cannot be larger than $\inf (\rho(p(A)) \  \rho(p(B)^{-1}))^{1/{deg(p)}}$.  As it trivially cannot be smaller,  (\ref{movingtospectral})  holds.

\end{proof}

It is of interest to know how small $\eta(A,B)$ can be.  Given a polynomially convex compact  set $K$ with positive logarithmic capacity, denote by $g$ the Green's function of the complement of $K$, with singularity at $\infty$. That is,  $g$ is harmonic in $\mathbb C\setminus K$, 
$$
g(z)= \log(z)+ O(1),  \ \text{ as }  \ z\rightarrow \infty
$$
and such that  for nearly everywhere on $\partial K$ $g(\zeta)\rightarrow 0$ as $\zeta$ tends to $\partial K$ from $\mathbb C\setminus K$, e.g. [13].   

\begin{theorem}\label{speedy}  Assume (\ref{firstcond}) holds and  $A$ is such that $\text{cap}(\widehat{\sigma(A)}) >0.$     Denote by $g$ the Green's function of  $\mathbb C \setminus \widehat{\sigma(A)}$.  Set $\alpha = \min_{\mu \in \sigma(B)} g(\mu)$.  Then  we have $0<\alpha <\infty$ and 
\begin{equation} 
\eta(A,B) =   e^{-\alpha}.
\end{equation}

\end{theorem}
\begin{proof} Here we use Bernstein's Lemma, as formulated in Theorem 5.5.7  of  [13]. Since $\sigma(B)$ and $\widehat{\sigma(A)}$ are  both compact, there is a positive distance between them  and  since $g$ is continuous and postive, we conclude $0< \alpha < \infty$.  Then Bernstein's Lemma yields for any polynomial $p$ of degree $d$
$$
\min_{\mu\in \sigma(B)} |p(\mu)|^{1/d} \le e^{ \alpha} \  \|p\|_{\sigma(A)}^{1/d} \  
$$
which means 
$$
\rho(p(B)^{-1})^{1/d} \ge e^{-\alpha} \  \rho(p(A))^{-1/d}.
$$
Thus
$$
\rho(p(B)^{-1})^{1/d}\rho(p(A))^{1/d} \ge e^{-\alpha}.
$$
To get $\eta(A,B)$ bounded from above we use  the following part of Theorem 5.5.7, [13]:  if $p$ is a Fekete polynomial  for $\widehat{\sigma(A)}$ of degree $d>1$, then 
$$
|p(z)| ^{1/d} \ge \|p\|_{\sigma(A)} ^{1/d} \ e^{g(z)} h(z,d)  \text{ for all } z \in \mathbb C \setminus \widehat{\sigma(A)}.
$$
Here $h$ is as follows:
$$
h(z,d)=\Big( \frac{ \text{cap} (\widehat{\sigma(A))}} {\delta_d(\widehat{\sigma(A))}} \Big)^{\tau(z)}
$$ where $\tau$ is the Harnack distance for $\mathbb C \setminus \widehat{\sigma(A)}$.  For us it suffices to know that $\tau$ is continuous and that
$$
\delta_n(K) \rightarrow \text{cap}(K) \ \text{ as } nÊ\rightarrow \infty.
$$
Thus,  for any $\varepsilon>0$ there exists a Fekete polynomial $p$  of degree $d$ such that
$$
\max_{\mu\in \sigma(B)}h(\mu,d) > \frac{1}{1+\varepsilon}.
$$
But then 
$$
\rho(p(B)^{-1})^{1/d} \le \rho(p(A))^{-1/d} e^{-\alpha} (1+\varepsilon).
$$
Multiplying this with $\rho(p(A))^{1/d}$ gives 
$$
\eta(A,B) \le \rho(p(B)^{-1})^{1/d} \  \rho(p(A))^{1/d} \le e^{-\alpha} (1+\varepsilon)
$$
which implies  the bound from above.
\end{proof}

Recall, that operators $A \in \mathcal B(\mathcal X)$ are called {\it quasialgebraic} if there exists a sequence  $\{p_j\}$ of monic polynomials such that
\begin{equation}\label{quasialg}
\inf \|p_j(A)\| ^{1/ deg(p_j)} =0.
\end{equation}
Halmos [3]  has shown that a bounded operator is quasialgberaic if and only if  the capacity of  its spectrum vanishes.  So, quasinilpotent, compact, polynomially compact, Riesz operators ect, are all quasialgebraic.   

\begin{theorem}  Let $A\in \mathcal B(\mathcal X)$, $B\in \mathcal B(\mathcal Y)$
satisfy $\widehat{\sigma(A)} \cap \sigma(B)=\emptyset$.  Then $\eta(A,B)=0$ if and only if $A$ is  quasialgebraic and, in particular $\widehat{\sigma(A)} = \sigma(A)$.
\end{theorem}
\begin{proof}
That $A$ being quasialgebraic is necessary, follows immediately from Theorem \ref{speedy}.  To obtain the other direction one needs to conclude that the superlinear decay  quaranteed for $A$ can be obtained with a sequence of polynomials with roots staying away from the spectrum of $B$.  This can be done for example by taking a nested sequence of compact sets $K_n$  such that $\cap K_n= \sigma(A)$, using Hilbert Lemniscate Theorem  to get polynomials  such that the associated lemniscates include $K_{n+1}$ but stay inside $K_n$.  The related Green's functions shall blow up at $\sigma(B)$.   

\end{proof}

\begin{remark}  In [8] we studied the polynomial acceleration speeds for the equation $ x=Lx+f$ with $L$ a bounded operator in a Banach space $\mathcal X$. 
We formulated the equation in the fixed point form, rather than the usual $Ax=b$,  to make the relationship between fixed point iteration and  e.g. Krylov methods more apparent.   Notice that  viewing $x$ and $f$ as  bounded operators $\mathbb C \rightarrow \mathcal X$, the fixed point equation can be viewed as a very special case of (\ref{original})  with $A=L$ and $B=1$.   The optimal asymptotic  convergence rate is, in agreement with the  results above, 
$$
\eta(A)= e^{- g(1)}
$$
provided $1 \notin \widehat{\sigma(A)}$, see Theorem 3.4.9 in [8]. Here $g$ denotes the Green's function when the capacity is positive and  can be thought as $+\infty$  when the capacity  vanishes.  We also discussed the  superlinear behavior when the capacity vanishes and  modelling the early behavior of iterations by assuming $1\in \partial{\widehat{\sigma(A)}}$ when  the speed is sublinear. 
\end{remark}


We now  derive a quantitative version of Theorem  1.4. 
 Denote by $S(A,B)$ again the mapping $X \mapsto AX-XB$.  Then the norm of  $S(A,B)^{-1}$  can be used to bound the perturbation sensitivity.  Since 
 $S(A,B)^{-1} = q(A,B) \circ S(p(A),p(B)) ^{-1}$  we have
 \begin{equation}\label{norminkasvu}
 \|S(A,B)^{-1}\| \le \|q(A,B)\| \|S(p(A),p(B)) ^{-1}\|.
 \end{equation}
When separating the operators using a polynomial $p$ the inversion should become easier  but  one would pay the prize of $q(A,B)$ typically having a large norm. However, as $q(A,B)$ is written out explicitly it can be thought of be applied exactly while the inversion   part  - when the dimensions are large or infinite -  would typically be done only approximatively, e.g. by truncating an iteration.  

It is tempting to replace the  separation condition $\widehat{\sigma(A)} \cap \sigma(B)= \emptyset $   by  the corresponding one  on pseudospectra:  
\begin{equation}\label{epsilonseparation}
\widehat{\Sigma_\varepsilon(A)}
\cap \Sigma_\varepsilon(B)= \emptyset,
\end{equation}
in  particular, as one of the the early applications of pseudospectrum was related to measuring the separation between matrices [15], [2].   However,  we shall rather use the following condition
\begin{equation}\label{puolittain}
V_p(A) \cap \Sigma_\varepsilon(B)=\emptyset
\end{equation}
which connects  the  polynomial $p$ directly into the estimates.  In  practice, one could  calculate $\Sigma_\varepsilon(B)$ with moderate $\varepsilon$ and search for  a polynomial $p$ e.g. by running  an Arnoldi type  Krylov process for a while and testing whether (\ref{puolittain}) is satisfied. This, or  even the "ideal Arnodi" method,  may not always  produce polynomials  with level set staying close to the spectrum.  In fact,  already the minimizing $\|p\|_K$  of monic polynomials of odd degree over $K=[-2,-1] \cup [1,2]$ necessarily has a zero at origin, staying far away from $K$.  For that reason the process behind  the proof of Theorem 1.7  is based on minimizing $\|p(A)\|$ over monic polynomials of given degree but includes a "cleaning" process - which most likely would not usually be needed.  Notice also, that if $\widehat{\Sigma_\varepsilon(A)}$ is known  and such that (\ref{epsilonseparation}) holds, then one could compute Fekete points on $\widehat{\Sigma_\varepsilon(A)}$ to get a polynomial for which (\ref{puolittain}) could hold.

Assume now that $\varepsilon$ and $p$ are such that (\ref{puolittain}) holds.  Then there exists $\delta>0$  and a contour $\gamma_B$ surrounding  $\Sigma_\varepsilon (B)$,   having vanishing total winding around $V_p(A)$,  and such  that  along $\gamma_B$  we have  $|p(\mu)| > \|p(A)\| + \delta$.  Let $\ell_B$ be the length of $\gamma_B$. 
Then 
$$
p(B)^{-k} = \frac{1}{2\pi i}\int_{\gamma_B} p(\mu)^{-k} (\mu - B)^{-1} d\mu
$$
which implies
$$
\|p(B)^{-k}\| \le \frac{\ell_B}{2\pi \varepsilon} (\|p(A)\|+\delta)^{-k}
$$
so that
\begin{equation}
\|p(A)^k\| \|p(B)^{-k-1}\|  \le \frac{\ell_B}{2\pi \varepsilon}  \frac{\|p(A)^k\|}{(\|p(A)\|+\delta)^{k+1}}.
\end{equation}
Summing up we have the following.

\begin{proposition} 
Assume that there is a polynomial $p$ and  $\varepsilon>0$  so that (\ref{puolittain}) holds. Then  with  $\delta$, $\ell_B$  as above  we have
\begin{equation}\label{kvantitatiivinen}
\|S(p(A),p(B))^{-1} \| \le  \frac{\ell_B}{2\pi \varepsilon} \sum_{k=0}^\infty \frac{\|p(A)^k\|}{(\|p(A)\|+\delta)^{k+1}}.
\end{equation}
\end{proposition}

\begin{remark}If $X=S(A,B)^{-1}(C)$ is wanted within some  tolerance, notice that  (\ref{norminkasvu}) and (\ref{kvantitatiivinen}) allow  one to calculate a safe truncation of the series expansion
$$
Y = \sum_{k=0}^\infty p(A)^k C p(B)^{-k-1}.
$$
In fact,  truncating 
$$
\widetilde Y=\sum_{k=0}^N p(A)^k C p(B)^{-k-1}
$$
and denoting $\widetilde X= q(A,B)(\widetilde Y)$
we  obtain $ \| \widetilde X-X\| < tol$,  providing $N$ is large enough so that
$$
r^{N+1} <  \frac{2\pi \varepsilon (1-r)}{{\ell_B} \|q(A,B)\|}   \ tol
$$
holds, where $r= \|p(A)\| / (\|p(A)\| + \delta)$.
\end{remark}


 \section {Half plane separation} 

The Theorem 1.6 deals with the question of existence of $p$ such that  
  the spectra are separated into different half planes, allowing one to solve the modified equation using  the integral representation (\ref{heinzintegral}) or the {\it sign-}function.

 Observe that
 \begin{equation}\label{triang}
 M= \begin{pmatrix} A&C\\
&B\end{pmatrix} =\begin{pmatrix} I&-X\\
&I\end{pmatrix}\begin{pmatrix} A\\
&B\end{pmatrix}\begin{pmatrix} I&X\\
&I\end{pmatrix}
\end{equation}
is satisfied exactly when $AX-XB=C$.   
If $\sigma(A) \subset \mathbb C_{+}$ and $\sigma(B)\subset \mathbb C_{-}$,    the {\it sign}-function is well defined at $M$ and we have 
\begin{equation}\label{signinkautta}
{\rm sgn}\begin{pmatrix} A&C\\
&B\end{pmatrix} =\begin{pmatrix} I&-X\\
&I\end{pmatrix}\begin{pmatrix} I \\
&-I\end{pmatrix}\begin{pmatrix} I&X\\
&I\end{pmatrix}=\begin{pmatrix} I&2X\\
&-I\end{pmatrix}.
\end{equation}
Thus,  $X$ can be obtained if  sgn$(M)$ can be computed.  This  is a rather popular route to compute the solution to Sylverster equation, see e.g. [1],  [6].  
 
We first prove the qualitative result of Theorem 1.6,  then discuss how the lemniscate set $V_p(A\oplus B)$ can be used  to obtain a quantitative result.


{\it Proof of Theorem 1.6}.   
The condition
$
\widehat{\sigma(A)} \cap \widehat{\sigma(B)} =\emptyset
$
is necessary.  In fact, assuming  (\ref{eripuolin}) holds, then we also have $\widehat{\sigma(p(A))} \subset\mathbb C_{+} $  { and }  $\widehat{\sigma(p(B))} \subset \mathbb C_{-}$ and hence
$$
\widehat{\sigma(p(A))} \cap \widehat{\sigma(p(B))}=\emptyset.
$$  
If $\lambda_0 \in \widehat{\sigma(A)} \cap \widehat{\sigma(B)}$  we get a contradiction as
$$
p(\lambda_0)  \in \widehat{\sigma(p(A))} \cap \widehat{\sigma(p(B))}. 
$$
 Here the last step follows from the general fact that if $z\in \widehat K$ and q is any polynomial, then  $|(q\circ p)(z)| \le \|q\circ p \|_K = \|q\|_{p(K)}$ and so, $p(z) \in \widehat{p(K)}$.

Assume therefore that (\ref{wideintersection}) holds and denote dist$(\sigma(A) , \sigma(B))= \delta$.  Put $U_1=  \{ \lambda \ :  { \rm dist}(\lambda, \widehat{\sigma(A)}) < \delta/3\}$ and $U_2=  \{ \mu \ :  { \rm dist}(\mu, \widehat{(\sigma(B)}) < \delta/3\}$.  Then denote  by $K$ the union of the closures of $U_1$ and $U_2$. 
Recall that $A(K)$ stands for  continuous functions in $K$ which are holomorphic in the interior of $K$.  Denote $c=$max$\{\|A\|,\|B\|\} +1$.  Then we define a function $\varphi \in A(K)$ as follows
\begin{equation}\label{holosep}
\varphi:   \overline U_1\ni z \mapsto z+c, \  \text{ while }    \overline U_2 \ni z  \mapsto z-c.
\end{equation} 
Since $\mathbb C \setminus K$ is connected we may by Mergelyan's Theorem approximate $\varphi$ arbitrarily accurately on $K$ by polynomials, say $\| \varphi - p\|_K< \varepsilon$. If $\gamma_1$ is a contour such that $\gamma_1$ surrounds $\widehat{\sigma(A)}$ inside $U_1$, then 
we have
$$
\|\varphi(A) -p(A)\| \le  \frac{\varepsilon}{2\pi} \int_{\gamma_1} \| (\lambda -A)^{-1} \| \ |d\lambda| 
$$
and in particular if $\varepsilon$ is small enough, $\sigma(p(A)) \subset \mathbb C_{+}.$  Defining  $\gamma_2$ in the similar way and integrating we get $p(B)$ with  spectrum in the left half plane.
 $\hfill\Box$
 \bigskip
 
We may replace the Mergelyan's Theorem in the proof of Theorem 1.6 by the use of multicentric representation of $\varphi$.
To that end, assume we have found polynomials $p_1$, $p_2$ such that
 \begin{equation}\label{erottelu}
 V_{p_1}(A) \cap V_{p_2}(B)= \emptyset,
 \end{equation}
 e.g. based on Theorem 1.7.  
Let then  $U_i$ be open, $ V_{p_1}(A)  \subset U_1$ and $ V_{p_2}(B)  \subset U_2$  and such that $\overline U_1\cap \overline U_2= \emptyset$.
 Then, again  by Theorem 1.7, we may assume that, applied to the block diagonal operator $A\oplus B$,  we have 
a polynomial $p$ such that 
\begin{equation}\label{algoritmilla}
V_p(A\oplus B) \subset U_1 \cup U_2.
\end{equation}
Without loss of generality we may assume that $p$ is of degree $d$ and has simple roots $\lambda_j$.
Let $t>0$ be small enough  so that   
$$
\gamma= \{\lambda \ : \  |p(\lambda)| = \|p(AÊ\oplus B)\| + t\} \subset U_1\cup U_2.
$$
Define $\varphi$ on  $\overline U_1\cup \overline U_2$ as in (\ref{holosep}).    We now use the multicentric representation (\ref{multi}) of $\varphi$ to  approximate $\varphi(A)$ and $\varphi(B)$ by  polynomials.  When $|w| < |p(\lambda)|$ we have
$$
 K_j(\lambda, w)=\frac{1}{\lambda-\lambda_j}   \sum_{n=0}^\infty w^n p(\lambda)^{-n}
 $$
and the functions $f_j$ in 
$$
\varphi(z)= \sum_{j=1}^d \delta_j(z) f_j(p(z))
$$
satisfy
$$
f_j(p(z)) = \frac{1}{2\pi i} \int_{\gamma}  K_j(\lambda,p(z))  \varphi(\lambda) d\lambda,
$$  
see [10]. We put
\begin{equation}\label{polkatko}
P(z)=\sum_{i=1}^d \delta_i(z)P_i(p(z))
\end{equation}
where we  truncate the series expansion for the integral kernel after the index $N$
$$
P_j(w)= \frac{1}{2\pi i} \int_\gamma \frac{\varphi(\lambda)}{\lambda-\lambda_j}   \sum_{n=0}^N w^n p(\lambda)^{-n} d\lambda,
$$
so that in particular $P$ is a polynomial of degree $(N+1)d -1$ at most.

Let $\gamma = \gamma_1 \cup \gamma_2$  with $\gamma_i \subset U_i$.     The roots of $p$ are divided into two parts, say $\lambda_j \in U_1$ for $j\le m$ and $\lambda_k \in U_2$ for $m<k \le d$.  Since $\sigma(A) \subset U_1$,  the integral over $\gamma_2$ does not contribute into $\varphi(A)$ and we may estimate as follows. Denote
$$
C_j =    \frac{1}{2\pi} \int_{\gamma} \frac{1}{|\lambda-\lambda_j| } |d\lambda|.
$$ 
Now
$$
P(A)= \sum_{j=1}^m \delta_j(A) \frac{1}{2\pi i} \int_{\gamma} \frac{\varphi(\lambda)}{\lambda-\lambda_j}   \sum_{n=0}^N p(A)^n p(\lambda)^{-n}  d\lambda
$$
and thus
$$
\| \varphi(A) -P(A) \| \le \|\varphi\|_\gamma \ \sum_{j=1}^m C_j \  \|\delta_j(A)\| \   \frac{1}{1-r } r^{N+1},
$$
where we   set $r=\frac{\|p(A)\|}{\|p(A\oplus B)\|+t}$.  Likewise  we obtain   
$$
\| \varphi(B) -P(B) \| \le \|\varphi\|_\gamma \  \sum_{k=m+1}^d C_k \  \|\delta_k(B))\| \   \frac{1}{1-s } s^{N+1},
$$
with $s=\frac{\|p(B)\|}{\|p(A\oplus B)\|+t}$.  By  the choice of $\varphi$  the spectrum of $\varphi(A)$ is in the half plane $Re\  \lambda >1$ while that of $\varphi(B)$ is likewise in the half plane $Re\  \mu < -1$.  Choosing $N$ large enough so that 
$$
\max \{ \| \varphi(A) -P(A) \|,\| \varphi(B) -P(B) \| \} <1
$$
 we have $\sigma(P(A)) \subset  \mathbb C_{+}$ and $\sigma(P(B)) \subset \mathbb C_{-}$.M
 
 To summarize:
 \begin{proposition}
Assume that we have  a polynomial $p$ such that (\ref{algoritmilla}) holds.   Then we can  estimate a truncation index $N$ such that  
\begin{equation}
\sigma(P(A)) \subset\mathbb C_{+}  \text { and } \ \sigma(P(B)) \subset \mathbb C_{-}
\end{equation}
holds with  the  polynomial $P$ in (\ref{polkatko}).  
\end{proposition}


\section{Explicit series expansion using multicentric calculus}

In the previous section we  demonstrated the existence  polynomials for half plane separation.  One could then  compute the {\it sign}-function of 
\begin{equation}\label{emm}
M= \begin{pmatrix} A&C\\&B\end{pmatrix}
\end{equation}  and obtain the solution $X$ to the Sylvester equation from (\ref{signinkautta}).   This can be  done  for example using Newton's iteration.   We shall here {\it bypass } the mapping into different half planes. We  use piecewise holomorphic functions  to define the  formal solution as a Cauchy-integral and then show how  using multicentric calculus  we get an explicit  series expression for it.  In the following we again assume all the time that $A\in \mathcal B(\mathcal X)$, $B\in \mathcal B(\mathcal Y)$ and $C\in \mathcal B(\mathcal Y, \mathcal X)$

Suppose we have open sets $U_1, U_2$ such that $\widehat{\sigma(A)} \subset U_1$ and $\widehat{\sigma(B)} \subset U_2$  and  $U_1\cap U_2= \emptyset$.  Let $\varphi$ be the locally constant holomorphic function taking value 1 in $U_1$ and value $-1$ in $U_2$.  If $\gamma_1$ is a contour inside $U_1$ surrounding $\sigma(A)$ we set
\begin{equation}\label{melkein}
Q= \frac{1}{2\pi i}\int_{\gamma_1} (\lambda-M)^{-1}.
\end{equation}
Then the following holds.
\begin{proposition} In the notation above  
\begin{equation}\label{ratkaisuun}
Q= \begin{pmatrix} I&X\\&0\end{pmatrix}
\end{equation}
where $X$ is the solution of $AX-XB=C$.

\end{proposition}
\begin{proof}
From (\ref{emm}) and 
 
 \begin{equation}
 M =\begin{pmatrix} I&-X\\
&I\end{pmatrix}\begin{pmatrix} A\\
&B\end{pmatrix}\begin{pmatrix} I&X\\
&I\end{pmatrix}
\end{equation}
 we have
 \begin{align}
 Q= & \begin{pmatrix} I&-X\\
&I\end{pmatrix}\frac{1}{2\pi i}\int_{\gamma_1} \begin{pmatrix} \lambda-A&\\ & \lambda-B\end{pmatrix}^{-1}\begin{pmatrix} I&X\\
&I\end{pmatrix}\\
=  &\begin{pmatrix} I&-X\\
&I\end{pmatrix}\begin{pmatrix}I&0 \\
0&0\end{pmatrix}\begin{pmatrix} I&X\\
&I\end{pmatrix}=\begin{pmatrix} I&X\\&0\end{pmatrix}.
\end{align}
\end{proof}

Our aim is now to compute $Q$. To that end let $\gamma= \gamma_1 \cup \gamma_2$ where $\gamma_2$ is a contour surrounding $\sigma(B)$ inside $U_2$ so that, as $\gamma$ surrounds $\sigma(M)$, we have
$$
I = \frac{1}{2\pi i}\int_{\gamma} (\lambda-M)^{-1}.
$$ 
But then  adding this to  both sides of 
$$
\varphi(M)= Q - \frac{1}{2\pi i}\int_{\gamma_2} (\lambda-M)^{-1}
$$
yields 
$
\varphi(M) = 2Q-I
$
and $Q=\frac{1}{2} ( \varphi(M) +I).$     Suppose  we have a polynomial $p$  such that
\begin{equation}\label{oletaemmalta}
V_p(M) \subset U_1 \cup U_2
\end{equation}
and  $t>0$ small enough so that   $\gamma = \{\lambda \ : |p(\lambda)| = \|p(M)\| +t \} \subset U_1 \cup U_2$. Then $\gamma$ splitts into  $\gamma_1$ and $\gamma_2$ in a natural way.  We now write down the series expansion of $\varphi$ which converge  inside $\gamma$, uniformly in  compact subsets.

On the polynomial $p$ we assume that it has simple roots and is monic and of degree $d$.    We write $\varphi$ in the multicentric form
\begin{equation}\label{kohtavalmis}
\varphi(\lambda) = \sum_{j=1}^d \delta_j(\lambda) f_j(p(\lambda))
\end{equation}  where the Taylor coefficients $\alpha_{j,k}$ in
$$
f_j(w) = \sum_{k=0}^\infty \alpha_{j,k} w^k
$$
can be computed by an explicit recursion.  The recursion is derived in [10].   Let  $p$ have roots $\lambda_j$ and $\delta_j(\lambda)$ denote the polynomials taking value 1 at $\lambda_j$ and vanishing at the other roots.  We may assume that $\lambda_j \in U_1$ for $j\le s$ and $\lambda_j\in U_2$ for $s+1\le j \le d$.  We  first compute recursively polynomials $b_{n,m}$ as follows:

Put $b_{0,0}=1$,  $b_{1,1}=p'$ , $b_{n,0} =0$ for $n>0$  and  for $m>n$ $b_{n,m}=0$.   Then 
 $$ 
 b_{n+1,m}= b_{n, m-1}  p' + b'_{n,m}.
$$
Then given the   values $\varphi^{(n)}(\lambda_j)$ we can compute  $f_j^{(n)}(0)$ from the following  
 \begin{align}
 (p'(\lambda_j))^n f_j^{(n)}(0)&=  
 \varphi^{(n)} (\lambda_j) \\
 &-\sum_{k=1}^d \sum_{m=0}^{n-1}  {n\choose m} \delta_k^{(n-m)}(\lambda_j)\sum_{l=0}^m b_{m, l}(\lambda_j)
f_k^{(l)} (0) \\ 
&-  \sum_{l=0}^{n-1} b_{n,l}(\lambda_j) f_j^{(l)}(0).
\end{align}
This is Proposition 4.3 in [10]\footnote{where the last line (5.11) had  dropped out}.  We can summarize:

\begin{proposition} Let $\varphi=1 $  in $U_1$ and  $\varphi=-1$ in $U_2$, and assume  $p$ is such that (\ref{oletaemmalta}) holds.  Then we have $Q= \frac{1}{2} (\varphi(M)+I)$ where
$$
\varphi(M) = \sum_{j=1}^d \delta_j(M) \sum_{n=0}^\infty \frac{f_j^{(n)}(0)}{n !} p(M)^n.
$$ 
\end{proposition}

The Taylor coefficients of $f_j$ satisfy, see Proposition 4.4 in [10],
\begin{equation}\label{fiinkaava}
\alpha_{j,n} =  \frac{f_j^{(n)}(0)}{n !}= \frac{1}{2\pi i} \int_\gamma \frac{\varphi(\lambda)}{p(\lambda)^n}\frac{d\lambda}{\lambda-\lambda_j}.
\end{equation}
Denote $L_j= \frac{1}{2\pi} \int_\gamma \frac{|d\lambda|}{|\lambda-\lambda_j|}$
then, we have
$
| \frac{f_j^{(n)}(0)}{n !}| \le L_j (\|p(M)\|+t)^{-n},
$
which allows us to truncate the series.  Put
$$
\widetilde{\varphi}(M) =\varphi(M) = \sum_{j=1}^d \delta_j(M) \sum_{n=0}^N \frac{f_j^{(n)}(0)}{n !} p(M)^n
$$
so that
\begin{equation}\label{fiinkatkaisu}
\|\widetilde{\varphi}(M) - \varphi(M)\| \le  \frac{C}{1-r}  r^{N+1}
\end{equation}
where
$$
C= \sum_{j=1}^d L_j \|\delta_j(M)\|,  \  \text{ and }  \ r=\frac{\|p(M)\|}{\|p(M)\|+t}.
$$
Let $tol>0$ be given and  compute $N$ such that
\begin{equation}\label{toler}
r^{N+1} < \frac{2(1-r)}{C} tol.
\end{equation}
 
\begin{proposition}
In the notation above, if $N$ is large enough so that (\ref{toler}) holds, then we  have an approximation $\widetilde X$ to $X$ solving $AX-XB=C$  such that $\|\widetilde X-X\| < tol$,  where $\widetilde X$ is the right upper corner element of $\widetilde Q= \frac{1}{2}(\widetilde {\varphi}(M) +I)$.
\end{proposition}

\begin{remark}  We may assume without loss of generality that $p$ has simple {\it rational } roots, as conditions such as (\ref{oletaemmalta}) allow small perturbations if needed. This means  that the Taylor coefficients $\alpha_{j,n}$ are rational as well.
\end{remark}
\begin{remark}
Observe that  we have an explicit formula for $p(M)^k$. In fact
$$
p(M)= \begin{pmatrix} p(A)&q(A,B)(C)\\
& p(B)\end{pmatrix}=: \begin{pmatrix} R&T\\
& S\end{pmatrix}
$$ and so
$$
p(M)^k =  \begin{pmatrix} R^k&q_{k-1}(R,S)(T)\\
& S^k\end{pmatrix}
$$
where $q_{k-1}(\lambda,\mu)= (\lambda^k - \mu^k)/(\lambda-\mu)$.
\end{remark}

\begin{example}  We shall  again demonstrate  the  approach using the special case as in Example 2.3.  Let $A$ and $B$ be nonsigular bounded operators in a Hilbert space, such that $A$ and $iB$ are self adjoint, normalized e.g. so that both have norms bounded by 1.  In particular then $A^2$  and $-B^2$ are both positive definite  with spectra in some interval $[\alpha,1]$, with $\alpha>0$.   We can proceed in two slightly different ways.

We could start by setting $\zeta= \lambda^2$ and  solve
\begin{equation}
A^2 X - X B^2= AC+CB
\end{equation}  
using {\it sign}-function expansion   in the polynomial $p(\zeta)=\zeta^2-1$. 
Or, you  could solve 
\begin{equation}
p(A) X-Xp(B)= q(A,B)(C)
\end{equation}
with $p(\lambda)= \lambda^4-1$ so that $q(\lambda,\mu)= \lambda^3+ \lambda^2 \mu + \lambda \mu^2 + \mu^3$.  Here you should define $\varphi =1$ in the  open sectors where  $\arg (\lambda^4 ) >0$ and $\varphi=-1$ where  $\arg(\lambda^4) <0$.   Both approaches lead to an expansion in terms of powers of $M^4-I$ which is easy to derive directly.  
Consider the  {\it sign}-function, defined for $Re \  \zeta \not=0$
as 
$$
\rm{sgn}(\zeta)= \frac{\zeta}{(\zeta^2 )^{1/2} }
$$
where $Re(\zeta^2 )^{1/2} >0$. With $w=\zeta^2-1$  and assuming that   $|w|= |\zeta^2-1|<1$  we may expand $(1+w)^{-1/2}$ to get
\begin{equation}
 {\rm sgn}(\zeta)= \zeta ( 1-\frac{1}{2} w + \frac{3}{8} w^2 - \frac{5}{16} w^3 + \cdots).
\end{equation}
Now,  we can apply this to the operator $M^2$. In fact, we have
$$
{\rm sgn}(M^2)= M^{2} ( I -\frac{1}{2} (M^4-1) + \frac{3}{8} (M^4-1)^2 - \frac{5}{16} (M^4-1)^3 + \cdots)
$$
which converges   as  the spectral radius  $\rho(M^4-I) = \|(A^2\oplus B^2)^2 -I\| <1$.   The solution to the original equation is  then the right upper corner element of $\frac{1}{2} {\rm sgn}(M^2)$.

\end{example}



\bigskip

\bigskip

 {\bf References}
 \smallskip
   
 [1] R. Bhatia, P. Rosenthal, How and why to solve the operator equation AX - XB = Y,  Bull. London Math. Soc., 29  (1997)1 - 21
   
  
[2] J. W. Demmel: The Condition Number of Equivalence Transformations that Block Diagonalize Matrix Pencils, SIAM J. Numer. Anal. 20, No. 3, (1983)
 
 [3]  P.R.Halmos, Capacity in Banach Algebras, IndianaUniv. Math., 20, (1971), pp 855-863, 
 
[4] Anders C. Hansen, Olavi Nevanlinna, Complexity issues in computing
 spectra, pseudospectra and resolvents, \'Etudes Op\'eratorielles,  Banach Center Publ. Volume 112, (2017), pp 171 - 194,  
  DOI: 10.4064/bc112-0-10
  
[5] E. Heinz, Beitr\"age zur St\"orungstheorie der Spektralzerlegung, Math. Ann. 123 (1951) 415 - 438.

[6]  N. J. Higham. Functions of matrices. Society for Industrial and Applied Mathematics
(SIAM), Philadelphia, PA, (2008)

 
[7]   Daniel Kressner, Bivariate Matrix Functions, Operators and Matrices, 8,2,  pp 449-466, (2014), DOI : 10.7153/oam-08-23
 
 [8]  Olavi Nevanlinna, Convergence of Iterations for Linear Equations, Birk\-h\"auser, (1993)
 
[9] Olavi Nevanlinna, Computing the spectrum and representing the resolvent, Numer. Funct. Anal. Optim. 30 no.9 - 10 (2009), 1025 - 1047.

 [10] O. Nevanlinna, Multicentric Holomorphic Calculus, Computational Methods and Function Theory, June 2012, Vol. 12, Issue 1, 45 - 65.

 [11]  O. Nevanlinna, Lemniscates and K -spectral sets, J. Funct. Anal. 262, (2012), 1728 - 1741.
 
 [12] O. Nevanlinna, Polynomial as a new variable - a Banach algebra with functional calculus, Oper. and Matrices 10 (3) (2016)  567 - 592
 
[13] T. Ransford, Potential Theory in the Complex Plane, London Math. Soc. Student Texts 28, Cambridge Univ. Press, (1995)

 [14]  M. Rosenblum, On the operator equation BX-XA=Q, Duke Math.J. 23 (1956) 263 - 270
 
 [15]    J.M.Varah, On the Separation of Two Matrices,  SIAM J. Numer. Anal., Vol. 16(2)  pp. 216 - 222, (1979)

\end{document}